\documentclass[a4paper,10pt]{article}

\usepackage{amsmath,amssymb,amsthm}
\usepackage{amssymb,latexsym, bbm,comment}

\numberwithin{equation}{section}

\newcommand{\origsetminus}{} \let\origsetminus=\setminus           
\renewcommand{\setminus}{\!\origsetminus\!}

\let\oldmarginpar\marginpar
\renewcommand\marginpar[1]{\-\oldmarginpar[\raggedleft\footnotesize #1]%
{\raggedright\footnotesize #1}}

{\theoremstyle{plain}
\newtheorem{lemma}{Lemma}[section]
\newtheorem{theorem}[lemma]{Theorem}

\newtheorem{proposition}[lemma]{Proposition}

} {\theoremstyle{definition}

\newtheorem{example}[lemma]{Example}
\newtheorem{remark}[lemma]{Remark}
}

\renewcommand{\mathbb}{\mathbbm}                     
\renewcommand{\epsilon}{\varepsilon}                 
\renewcommand{\phi}{\varphi}
\renewcommand{\theta}{\vartheta}
\renewcommand{\le}{\leqslant}
\renewcommand{\ge}{\geqslant}

\newcommand{\origfoo}{} \let\origfoo=\sqrt           
\renewcommand{\sqrt}[1]{\origfoo{#1}\;}

\newcommand{\abs}[1]{\left\lvert #1 \right\rvert}    
\newcommand{\norm}[1]{\left\lVert #1 \right\rVert}   
\DeclareMathOperator{\R}{{\mathbb R}}                

\DeclareMathOperator{\C}{{\mathbb C}}                
\DeclareMathOperator{\N}{{\mathbb N}}                
\newcommand{\A}{{\mathcal A}}

\DeclareMathOperator{\Borel}{{\mathcal B}}
\newcommand{\scapro}[2]{\langle #1,#2\rangle}       
\DeclareMathOperator{\1}{\mathbbm 1}

\renewcommand{\L}{{\mathcal L}}

\newcounter{zahl}

\DeclareMathOperator{\Cc}{{\mathcal C}} \DeclareMathOperator{\Rad}{{\mathcal R}}

\DeclareMathOperator{\Z}{{\mathcal Z}}

\title{Radonifying operators \\and \\infinitely divisible Wiener integrals}

\author{ Markus Riedle\footnote{The author acknowledges the EPSRC grant EP/I036990/1}\\
Department of Mathematics\\
King's College London\\
London WC2R 2LS\\
United Kingdom
 }

\begin{document}

\maketitle

\begin{center}
\begin{minipage}[t]{11cm}
\begin{center}
   Dedicated to the 50-th anniversary of the\\
Department of Theory of Random Processes\\
Institute of Mathematics National Academy of Sciences of
Ukraine
\end{center}
\end{minipage}
\end{center}
\vspace*{1cm}

\begin{abstract}
In this article we illustrate the relation between the existence of  Wiener integrals with respect to a L\'evy process in a separable Banach space and radonifying operators. For this purpose, we introduce the class of $\theta$-radonifying operators, i.e.\ operators which map a cylindrical measure $\theta$ to a genuine Radon measure. We study this class of operators for various examples of infinitely divisible cylindrical measures $\theta$ and highlight the differences from the Gaussian case.

\end{abstract}

\section{Introduction}

Starting with the work by Gel'fand \cite{Gelfand}, Gross \cite{Gross} and Segal \cite{Segal} the canonical Gaussian cylindrical  measure has gained much attention in different areas of mathematics and applications.
It is not only of interest from a theoretical point of view but it is also of importance in various applications such as filtering problems in Bensoussan \cite{Bensoussan}, small ball probabilities in Li and Linde \cite{LiLinde},  interest rate models in Carmona and Tehranchi \cite{Carmona} and stochastic integration in Banach spaces in van Neerven, Veraar and Weis \cite{vanNeervenetal07}.

In his seminal work \cite{Gross} Gross studies norms on a Hilbert space $H$ such that the canonical Gaussian cylindrical measure $\gamma$ extends to a $\sigma$-additive probability measure on the completion of $H$ with respect to the norm. This directly
leads to the class $\Rad(\gamma)$ of $\gamma$-radonifying operators, which consists of linear and bounded operators $T$ from $H$ to a Banach space $V$ such that the cylindrical  image measure $\gamma\circ T^{-1}$ extends to a $\sigma$-additive probability measure. The space $\Rad(\gamma)$ is known to have many desirable properties, such as the completeness under an appropriate norm, the ideal property and close relations to absolutely summing operators.

Recently, the space of $\gamma$-radonifying operators plays a fundamental role
in the theory of stochastic integration in Banach spaces. In \cite{vanNeervenetal07}, van Neerven, Veraar and Weis  develop a theory of stochastic integration for random operator-valued integrands with respect to cylindrical Wiener processes in UMD Banach spaces. Their approach is strongly based on the corresponding Wiener integrals for deterministic integrands introduced in \cite{BrzezniakvanNeerven00} and  \cite{vanNeervenWeis}, and those existence is naturally closely related to the class of $\gamma$-radonifying operators.

In our  work \cite{OnnoMarkus}, we extend the approach in \cite{vanNeervenWeis} to Wiener integrals for deterministic integrands with respect to martingale-valued measures, in particular to L\'evy processes. However, this work \cite{OnnoMarkus} was accomplished under the constraint not being able to use any of the fundamental properties of the space of $\gamma$-radonifying operators since the analogue theory was not developed in a non-Gaussian setting. It became apparent, that if one would like to develop a theory of stochastic integration for random integrands similarly to the one in \cite{vanNeervenetal07} but for L\'evy processes, one needs to study an analog class of operators as $\gamma$-radonifying operators but radonifying an infinitely divisible cylindrical  measure. This is the main motivation of this work where we show that one can introduce such a space of operators although it lacks many of the fundamental properties of $\gamma$-radonifying operators.

The canonical Gaussian cylindrical measure $\gamma$ is distinguished among all Gaussian cylindrical measures by its characteristic function, which most often serves also as its definition. Equivalently, starting from a Gaussian cylindrical random variable $X$ in an arbitrary Banach space $V$ with covariance operator $Q$, one can explicitly construct a cylindrical random variable $\Theta$ in the reproducing kernel Hilbert space of $Q$ whose cylindrical distribution equals the canonical Gaussian cylindrical distribution $\gamma$. This construction is based on the Karhunen-Lo{\`e}ve expansion of $X$. We show in the first part of this work, that this construction of a canonically Gaussian distributed cylindrical random variable $\Theta$ on the reproducing kernel Hilbert space can be mimicked for each cylindrical random variable with second moments. Denoting the cylindrical distribution of $\Theta$ by $\theta$, this construction motivates us to define the class $\Rad(\theta)$ of $\theta$-radonifying operators in analogy to $\gamma$-radonifying operators as the space of operators $T$ such that the image cylindrical measure $\theta\circ T^{-1}$ extends to a $\sigma$-additive probability measure on the Borel $\sigma$-algebra.

The class $\Rad(\theta)$ of $\theta$-radonifying operators is only well studied if $\theta$ equals the canonical Gaussian cylindrical measure $\gamma$ or a canonical stable
cylindrical measure. In this work we show that for an arbitrary cylindrical measure $\theta$ the linear space $\Rad(\theta)$ can be equipped with a certain norm, introduced in this  work,  such that it becomes complete.  However, already in the case of a canonical stable cylindrical measure $\theta$, which might be considered as a non-Gaussian cylindrical measure most similar to  the canonical Gaussian cylindrical measure $\gamma$, it is known that the space $\Rad(\theta)$ lacks many of the desirable properties of the space of $\gamma$-radonifying operators.  We study the linear space  $\Rad(\theta)$ for different examples of  infinitely divisible cylindrical measures $\theta$ and compare it to the Gaussian situation.

In the last part of this work, we illustrate the relation of $\theta$-radonifying operators and the existence of Wiener integrals with respect to a L\'evy process. Although this is  the underlying idea in the work
\cite{vanNeervenWeis} and to some extent in the generalisation \cite{OnnoMarkus}, we are able to illustrate this relation more explicitly by defining a {\em cylindrical} integral, which in the case of stochastic integrability is induced by a genuine Banach space valued random variable. In particular for  L\'evy driven integrals, this rigorous relation between stochastic integrability and Banach space valued  operators is novel, and it signficantly improves the description of integrable operators in \cite{OnnoMarkus}.

\section{Preliminaries}\label{se.preliminaries}

Throughout this paper, $V$ is a separable Banach space with dual $V^\ast$ and dual pairing  $\scapro{\cdot}{\cdot}$. The Borel $\sigma$-algebra is denoted by $\Borel(V)$.
If $U$ is another separable Banach space the space of bounded and linear operators is denoted by $\L(U,V)$ equipped with the uniform operator norm $\norm{\cdot}_{U\to V}$. An operator $T\in \L(U,V)$ is called {\em $p$-absolutely summing}
if there exists a constant $c>0$ such that for each $n\in\N$ and $u_1,\dots, u_n\in U$ it obeys
\begin{align}\label{eq.def-summable}
  \sum_{k=1}^n \norm{Tu_k}^p\le c^p \sup_{\norm{u^\ast}\le 1}\sum_{k=1}^n \abs{\scapro{u_k}{u^\ast}}^p.
\end{align}
The space of all $p$-absolutely summing operators is denoted by $\Pi^p(U,V)$ and it is a  Banach space
 under the norm $\norm{T}_{\Pi^p}:=\pi_p(T)$ where $\pi_p(T)$ is the smallest constant $c$ satisfying \eqref{eq.def-summable}.

For a measurable space $(S,{\mathcal S}, m)$ and $p\ge 1$ we shall denote the Lebesgue-Bochner space by $L^p_m(S;V)$. A probability space is denoted by $(\Omega,\A,P)$ and $L^0_P(\Omega;\R)$ denotes the space of equivalence classes of measurable functions equipped with the topology of
convergence in probability.

For every $v^\ast_1,\dots, v^\ast_n\in V^{\ast}$ and $n\in\N$ we define a linear map
\begin{align*}
  \pi_{v^\ast_1,\dots, v^\ast_n}\colon V\to \R^n,\qquad
   \pi_{v^\ast_1,\dots, v^\ast_n}(v)=\big(\scapro{v}{v^\ast_1},\dots,\scapro{v}{v^\ast_n}\big).
\end{align*}
For $n\in\N$ and $B\in \Borel(\R^n)$, sets of the form
 \begin{align*}
C(v^\ast_1,\dots ,v^\ast_n;B):&= \{v\in V:\, (\scapro{v}{v^\ast_1},\dots,
 \scapro{v}{v^\ast_n})\in B\}
 = \pi^{-1}_{v^\ast_1,\dots, v^\ast_n}(B)
\end{align*}
are called {\em cylindrical sets}. If $D$ is a subset of $V^\ast$ then
\begin{align*}
  \Z(V,D):=\left\{\pi_{v_1^\ast,\dots, v_n^\ast}^{-1}(B):\, v_1^\ast,\dots, v_n^\ast\in D,\,
     B\in \Borel(\R^n),\, n\in \N\right\},
\end{align*}
defines the {\em cylindrical algebra generated by $D$}. The generated $\sigma$-algebra is denoted by $\Cc(V,D)$ and it is called the {\em cylindrical $\sigma$-algebra with
respect to $(V,D)$}. If $D=V^\ast$ we write $\Z(V):=\Z(V,D)$ and $\Cc(V):=\Cc(V,D)$.

A function $\eta\colon \Z(V)\to [0,\infty]$ is called a {\em cylindrical measure on
$\Z(V)$} if for each finite subset $D\subseteq V^\ast$ the restriction of
$\eta$ to the $\sigma$-algebra $\Cc(V,D)$ is a measure. A cylindrical
measure $\eta$ is called finite if $\eta(V)<\infty$ and a cylindrical probability
measure if $\eta(V)=1$. The characteristic function $\phi_\eta$ of a
finite cylindrical measure $\eta$ is defined by
\begin{align*}
\phi_\eta\colon V^\ast\to\C,\qquad \phi_{\eta}(v^\ast):=\int_V e^{i\scapro{v}{v^\ast}}\,\eta(dv).
\end{align*}
We will always assume that the characteristic function is continuous, in which case the cylindrical measure $\eta$ is called {\em continuous}. A cylindrical measure $\eta$ has $p$-th weak moments if
\begin{align*}
  \int_V \abs{\scapro{v}{v^\ast}}^p\,\eta(dv)<\infty\qquad\text{for all }
  v^\ast\in V^\ast.
\end{align*}
A cylindrical measure $\eta$ is of cotype $p$ if it has $p$-th weak moments and
for each sequence $(v_n^\ast)_{n\in\N}\subseteq V^\ast$ the condition
\begin{align*}
\int_V \abs{\scapro{v}{v_n^\ast}}^p\,\eta(dv) \to 0 \quad\text{for }n\to\infty,
\end{align*}
implies that $\norm{v_n^\ast}\to 0$.

A {\em cylindrical random variable $Z$ in $V$} is a linear and continuous map
\begin{align*}
 Z\colon V^\ast \to L^0_P(\Omega;\R).
\end{align*}
The cylindrical random variable $Z$ has weak $p$-th moments if $E[\abs{Zv^\ast}^p]<\infty$ for all
$v^\ast\in V^\ast$. In this case, the closed graph theorem implies that $Z\colon V^\ast \to L^p_P(\Omega;\R)$ is continuous.
The characteristic function of a cylindrical random  variable $Z$ is
defined by
\begin{align*}
 \phi_Z\colon V^\ast \to\C, \qquad \phi_Z(v^\ast)=E\big[\exp(iZv^\ast)\big].
\end{align*}
By defining for each cylindrical set $C=C(v_1^\ast,\dots, v_n^\ast;B)\in \Z(V)$ the mapping
\begin{align*}
  \eta_Z(C):=P\big((Zv^\ast_1,\dots, Zv^\ast_n)\in B\big),
\end{align*}
we obtain a cylindrical probability measure $\eta_Z$, which is called the {\em cylindrical distribution of $Z$}. The characteristic functions $\phi_{\eta_Z}$ and $\phi_Z$ of $\eta_Z$ and $Z$ coincide. Conversely, for every cylindrical probability measure $\eta$ on
$\Z(V)$ there exist a probability space $(\Omega,\A,P)$ and a
cylindrical random variable $Z\colon V^\ast\to L^0_P(\Omega;\R)$ such
that  $\eta$ is the cylindrical distribution of $Z$; see \cite[VI.3.2]{Vaketal}.

A cylindrical random variable $Z\colon V^\ast\to L^0_P(\Omega;\R)$ is called {\em induced by
a random variable in $L^p_P(\Omega;V)$} if there exists $Y\in L^p_P(\Omega;V)$ such that
\begin{align*}
  \scapro{Y}{v^\ast}= Zv^\ast\qquad\text{for all }v^\ast \in V^\ast.
\end{align*}
This is equivalent to the fact that the cylindrical distribution of $Z$ extends to a probability measure on $\Borel(V)$; see Theorem IV.2.5 in \cite{Vaketal}.

\section{Infinitely divisible cylindrical measures}\label{se.infdiv}

The class of infinitely divisible cylindrical probability measures is introduced in \cite{Riedle11}.
A cylindrical probability measure $\eta$ on $\Z(V)$ is called {\em infinitely divisible}
if  for each $k\in\N$ there exists a cylindrical  probability measure $\eta_k$ such that
$\eta=\eta_k^{\ast k}$. Theorem 3.13 in \cite{Riedle11} shows that a cylindrical probability measure $\eta$ is infinitely divisible if and only if
\begin{align*}
  \eta\circ \pi_{v_1^\ast,\dots, v_m^\ast}^{-1}
  \text{ is infinitely divisible on $\Borel(\R^m)$
 for all $v_1^\ast,\dots, v_m^\ast\in V^\ast$ and $m\in\N$.}
\end{align*}
In this equivalent description it is not sufficient only to take $n=1$ as it is shown even in the case $V=\R^2$ in \cite{Marcus} and \cite{GineHahn}.

Let $X$ be an infinitely divisible cylindrical random variable, that is its cylindrical distribution is infinitely divisible, and assume that $X$ has weak second moments and $E[Xv^\ast]=0$ for all $v^\ast\in V^\ast$. Define the covariance operator by
\begin{align*}
  Q\colon V^\ast\to V^{\ast\ast}, \qquad
  \scapro{Qv^\ast}{w^\ast}=E\big[(Xv^\ast)(Xw^\ast)\big].
\end{align*}
The range of $Q$, i.e. the continuity of $Qv^\ast\colon V^\ast\to \R$, follows
from the Cauchy-Schwarz inequality and the continuity of $X\colon V^\ast\to L^2_P(\Omega;\R)$.

For the following we assume that $Q$ is $V$-valued. This is guaranteed for example if
 $X$ is a genuine random variable (see Theorem III.2.1 in \cite{Vaketal}), in which case we set $Xv^\ast=\scapro{X}{v^\ast}$. Other examples of a $V$-valued covariance operator will be seen later in Section \ref{se.Application}. As the covariance operator $Q\colon V^\ast\to V$ is positive and symmetric, it follows that there exists a Hilbert space $H$ and $j\in L(H, V)$ such that
$Q=j j^\ast$ and $j(H)$ is dense in $V$. Moreover, the Hilbert space $H$ is unique up to isomorphism and separable as $V$ is separable; see Section III.1.2 in \cite{Vaketal}.

Since the range of $j^\ast$ is dense in $H$ we can choose an orthonormal basis $(e_k)_{k\in\N}$ of $H$ with $e_k\in j^\ast(V^\ast)$. Thus, there exist some elements $v_k^\ast \in V^\ast $ obeying $j^\ast v_k^\ast=e_k$ for all $k\in\N$. It follows
that
\begin{align}\label{eq.Loeve}
  Xv^\ast= \sum_{k=1}^\infty \scapro{je_k}{v^\ast} Xv_k^\ast
  \qquad\text{for all }v^\ast\in V^\ast,
\end{align}
where the sum converges in $L^2_P(\Omega;\R)$.
This representation is the Karhunen-Lo{\`e}ve representation and is in this form
established in \cite{ApplebaumRiedle}. Define a cylindrical random variable by
\begin{align}\label{eq.canZ}
  \Theta_X\colon H\to L^2_P(\Omega;\R), \qquad \Theta_Xh= \sum_{k=1}^\infty \scapro{e_k}{h} Xv_k^\ast.
\end{align}
The fact, that $\Theta_X$ is well defined and is a cylindrical random variable follows from the following lemma where we collect some simple properties of $\Theta_X$ and its cylindrical probability distribution $\theta_X$.
\begin{lemma}\label{le.properties}
For a cylindrical random variable $X\colon V^\ast\to L^2_P(\Omega;\R)$ let $\theta_X$ denote the cylindrical distribution of $\Theta_X$
defined in \eqref{eq.canZ}. Then we have:
\begin{enumerate}
  \item[\rm (a)] $\displaystyle E\left[\abs{\Theta_Xh}^2\right]  =\norm{h}^2$ for all $h\in H$.
  \item[\rm (b)] the cylindrical distribution of $X$ equals $\theta_X\circ j^{-1}$.
  \item[\rm (c)] $\theta_X$ is of cotype 2.
  \item[\rm (d)] $\theta_X$ is infinitely divisible.
\end{enumerate}
\end{lemma}
\begin{proof}
(a) The identity
\begin{align}\label{eq.XCov}
  E[(Xv_k^\ast)(Xv_\ell^\ast)]=\scapro{Qv_k^\ast}{v_\ell^\ast}=\scapro{e_k}{e_\ell}
\end{align}
implies that $Xv_k^\ast$ and $Xv_\ell^\ast$ are uncorrelated for $k\neq \ell$ and $E[\abs{Xv_k^\ast}^2]=1$. Thus, part (a) follows
from \eqref{eq.canZ}. This also shows that $\Theta_X$ is a well defined cylindrical random variable. \\
(b) Due to \eqref{eq.Loeve} and \eqref{eq.canZ} we have $\Theta_X(j^\ast v^\ast)=Xv^\ast$ for all $v^\ast\in V^\ast$ which establishes the claim. Part (c) follows from part (a). \\
(d) For $h_1,\dots, h_m\in H$ and  $n\in\N$ define
\begin{align*}
  u_j^{(n)}:=\sum_{k=1}^n \scapro{e_k}{h_j}v_k^\ast\qquad\text{for }j=1,\dots, m.
\end{align*}
It follows from \eqref{eq.canZ} for each $j=1,\dots, m$ that
\begin{align*}
  \Theta_X h_j=\lim_{n\to\infty} Xu_j^{(n)}
  \quad\text{in }L^2_P(\Omega;\R),
\end{align*}
which yields
\begin{align*}
 (\Theta_Xh_1, \ldots , \Theta_Xh_m)=\lim_{n\to\infty} \Big( Xu_1^{(n)}, \ldots, Xu_m^{(n)}\Big)
 \qquad\text{in probability in $\R^m$.}
\end{align*}
The probability distribution of the random vector  on the right hand side is infinitely divisible as it is given  by
$\eta\circ \pi_{u_1^{(n)},\dots, u_m^{(n)}}^{-1}$, where $\eta$ denotes the cylindrical distribution of the infinitely divisible cylindrical  random variable $X$. Consequently, the random vector on the left hand side is infinitely divisible, which shows that $\Theta_X$ is an infinitely divisible cylindrical random variable.
\end{proof}

\begin{example}\label{ex.canonicalGauss}
  Assume that $X$ is a Gaussian cylindrical random variable that is $Xv^\ast$ is Gaussian for all
  $v^\ast\in V^\ast$. In this case, it  follows from \eqref{eq.XCov}
  that $(Xv_k^\ast)_{k\in\N}$ is a sequence of independent, Gaussian random variables with
  $E[\abs{Xv_k^\ast}^2]=1$, which yields for the characteristic function $\phi_{\Theta_X}$
  of $\Theta_X$:
  \begin{align}\label{eq.char-gamma}
    \phi_{\Theta_X}(h)
    =\prod_{k=1}^\infty \exp\left(-\tfrac{1}{2} \scapro{e_k}{h}^2\right)
    =\exp\left(-\tfrac{1}{2}\norm{h}^2\right)\qquad\text{for all }h\in H.
  \end{align}
Consequently, the cylindrical random variable $\Theta_X$ is distributed according to the
canonical Gaussian cylindrical  measure $\gamma$ in this case.
\end{example}

\section{Radonifying operators}

Let $H$ be a separable Hilbert space, $V$ be a separable Banach space and $\theta$  be a cylindrical probability measure on $\Z(H)$. An operator $T\in \L(H,V)$ is called {\em $\theta$-radonifying}
if the image cylindrical measure $\theta\circ T^{-1}$ extends to a probability measure on $\Borel(V)$.
If the extended measure has finite $p$-th moments, $T$ is called  $\theta$-radonifying of order $p$. We define the space
\begin{align*}
  \Rad^p (\theta):=\Rad^p_{H,V}(\theta):=\big\{T\in \L(H,V):\,
  T\text{ is $\theta$-radonifying of order $p$}\big\}.
\end{align*}
Let $\Theta$ denote a cylindrical random variable with cylindrical distribution
$\theta$. Theorem VI.3.1 in \cite{Vaketal} guarantees that an operator $T\in \L(H,V)$ is
in $\Rad^p (\theta)$ if and only if the cylindrical random variable
$T(\Theta)$ defined by
\begin{align*}
  T(\Theta)\colon V^\ast \to L^0_P(\Omega;\R),\qquad
   T(\Theta)v^\ast=\Theta(T^\ast v^\ast)
\end{align*}
is induced by a genuine random variable in $L^p_P(\Omega;V)$.
If $S$ and $T$ are in $\Rad^p(\theta)$ and $\alpha\in\R$ then we obtain for all $v^\ast\in V^\ast$
that
\begin{align*}
(\alpha S+ T)(\Theta)v^\ast=  \Theta ((\alpha S+ T)^\ast v^\ast)=  \alpha S(\Theta)v^\ast + T(\Theta)v^\ast.
\end{align*}
Since $S(\Theta)$ and $T(\Theta)$ are induced by  genuine random variables in $L^p_P(\Omega;V)$, respectively, it follows that $(\alpha S+ T)(\Theta)$ is also induced by a genuine random variable in $L^p_P(\Omega;V)$, and thus $\alpha S+T\in \Rad^p(\theta)$. For $T\in \Rad^p(\theta)$ define
\begin{align*}
  \norm{T}_{p}:=\left(\int_V \norm{v}^p \, (\theta\circ T^{-1})(dv)\right)^{1/p}.
\end{align*}
Since $\norm{T}_p^p=E[\norm{T(\Theta)}^p]$ it follows that $\norm{\cdot}_p$ defines a semi-norm on $\Rad^p(\theta)$.
We obtain a norm on $\Rad^p(\theta)$ by defining
\begin{align*}
  \norm{T}_{\Rad^p}:=\norm{T}_p+ \norm{T}_{H\to V}.
\end{align*}

\begin{theorem}\label{th.Banach}
For every cylindrical probability measure $\theta$ on $\Z(H)$ the
space $\Rad^p(\theta)$ equipped with $\norm{\cdot}_{\Rad^p}$ is a Banach space for each $p\ge 1$.
\end{theorem}

\begin{proof}
Let $(T_n)_{n\in\N}$ be a Cauchy sequence in $\Rad^p(\theta)$. Denote for each $n\in\N$ by $Y_n$ the induced random variable in $L^p_P(\Omega;V)$ with probability distribution $\theta\circ T_n^{-1}$. Since $(T_n)_{n\in\N}$ is also a Cauchy sequence in $\L(H,V)$ it follows that there exists  $T\in \L(H,V)$ such that $\norm{T_n-T}_{H\to V}\to 0$ for $n\to\infty$.  Moreover, the equality $\norm{Y_m-Y_n}_{L^p_P}=\norm{T_m-T_n}_{p} $ implies that there exists a random variable $Y \in L^p_P(\Omega;V)$ such that $\norm{Y_n-Y}_{L^p_P}\to 0$ for
$n\to\infty$.
The continuity of $\Theta\colon H\to L_P^0(\Omega;\R)$ implies for every $v^\ast$ that we have
in $L_P^0(\Omega;\R)$:
\begin{align*}
  \lim_{n\to\infty} \abs{\scapro{Y_n}{v^\ast}-\Theta(T^\ast v^\ast)}
  =\lim_{n\to\infty} \abs{\Theta\big(( T_n^\ast-T^\ast)v^\ast\big)}=0.
\end{align*}
Since $\scapro{Y_n}{v^\ast}\to \scapro{Y}{v^\ast}$ in $L_P^2(\Omega;\R)$ we obtain
$\scapro{Y}{v^\ast}=\Theta(T^\ast v^\ast)$ for all $v^\ast\in V^\ast$, which completes the proof.
\end{proof}
\begin{example}
  Let $\theta$ be given by the canonical Gaussian cylindrical measure $\gamma$ on $\Z(H)$. Due to Fernique's theorem, each $\gamma$-radonifying operator is of any order $p\ge 1$. Thus, the space $\Rad^p(\gamma)$ coincides with the space of $\gamma$-radonifying operators. This class of operators is well studied, and is recently surveyed in \cite{vanNeerven10}.
\end{example}
In the special setting of Section \ref{se.infdiv} we obtain the following simplification of the norm in $\Rad^p(\theta)$ for $p\ge 2$:
\begin{proposition}
For a cylindrical random variable $X$ in $V$ with weak second moments let the cylindrical random variable  $\Theta_X$ and its cylindrical distribution $\theta_X$ be defined by \eqref{eq.canZ}.
If $p\ge 2$ then $T\in \Rad^p(\theta_X)$ satisfies
\begin{align*}
\norm{T}_{H\to V}\le  \norm{T}_p,
\end{align*}
that is the norm $\norm{\cdot}_{\Rad^p}$ is equivalent to $\norm{\cdot}_p$.
\end{proposition}
\begin{proof}
For $T\in \Rad^p(\theta_X)$ let $Y$ denote the induced random variable in $ L_P^p(\Omega;V)$ with probability distribution $\theta_X\circ T^{-1}$ on $\Z(V)$.
Lemma \ref{le.properties} implies:
\begin{align*}
  \norm{T}_{H\to V}^2
= \sup_{\norm{v^\ast}\le 1} \norm{T^\ast v^\ast}^2
&= \sup_{\norm{v^\ast} \le 1} E\left[\abs{\Theta_X T^\ast v^\ast}^2\right]\\
&\le E\left[ \sup_{\norm{v^\ast} \le 1} \abs{\scapro{Y}{v^\ast}}^2 \right]
= E\left[\norm{Y}^2\right]
\le\Big( E\Big[\norm{Y}^p\Big]\Big)^{2/p},
\end{align*}
which completes the proof.
\end{proof}

The following result is a straightforward conclusion of a result by Schwartz \cite{Schwartz69} and Kwapien \cite{Kwapien70}, but it shows an important class of operators which are $\theta$-radonifying.
\begin{proposition}\label{pro.abs-rad}
If $\theta$ is a cylindrical probability measure of weak order $p\ge 1$ then we have
\begin{align*}
  \Pi^p(H,V)\subseteq \Rad^p_{H,V}(\theta).
\end{align*}
\end{proposition}
\begin{proof}
The space of $p$-absolutely summing operators $\Pi^p(H,V)$  coincides with the space of $p$-radonifying
operators; see Theorem VI.5.4 in \cite{Vaketal} for $p>1$ and Corollary in VI.5.4 in \cite{Vaketal} for $p=1$.
The space of $p$-radonifying operators are operators $T\in \L(H,V)$ such for each cylindrical measure $\eta$ on $\Z(H)$ with weak $p$-moments the image cylindrical measure $\eta\circ T^{-1}$ extends to
a measure on $\Borel(V)$ with finite $p$-moment.
\end{proof}

\begin{proposition}\label{pro.radonifying-sum}
Let $\Theta\colon H\to L_P^0(\Omega;\R)$ be a cylindrical random variable with cylindrical distribution
$\theta$.  Then for  $T\in \L(H,V)$ and $p\ge 1$ the following are equivalent:
\begin{enumerate}
\item[\rm (a)] $T\in \Rad^p(\theta)$.
  \item[\rm (b)] there exists a random variable $Y\in L_P^p(\Omega;V)$
  such that for all (some) orthonormal basis $(e_k)_{k\in\N}$ of $H$:
\begin{align*}
  \scapro{Y}{v^\ast}=\sum_{k=1}^\infty \scapro{Te_k}{v^\ast} \Theta e_k
  \quad\text{in $L^0_P(\Omega;\R)$ for all }v^\ast\in V^\ast.
\end{align*}
\end{enumerate}
\end{proposition}
\begin{proof}
The operator $T$ is in $\Rad^p(\theta)$ if and only if the
cylindrical random variable $T(\Theta)$ is induced by a $V$-valued random variable $Y$
in $L_P^p(\Omega;V)$, that is $\Theta(T^\ast v^\ast)=\scapro{Y}{v^\ast}$ for all $v^\ast\in V^\ast$.
The continuity of $\Theta\colon H\to L_P^0(\Omega;\R)$ implies in $L_P^0(\Omega;\R)$:
\begin{align*}
 \scapro{Y}{v^\ast}=
  \Theta (T^\ast v^\ast)=
  \Theta\left(\sum_{k=1}^\infty \scapro{T^\ast v^\ast}{e_k}e_k\right)
 =  \sum_{k=1}^\infty  \scapro{v^\ast}{Te_k}\Theta e_k
\end{align*}
for all $v^\ast\in V^\ast$.
\end{proof}

\begin{remark}
If $\theta$ is the canonical Gaussian cylindrical measure $\gamma$ on $H$, then the random variables $(\Theta e_k)_{k\in\N}$ are independent and symmetric. Thus, It{\^o}-Nisio's Theorem guarantees that part (b) in Proposition
\ref{pro.radonifying-sum} is equivalent to
\begin{enumerate}
  \item[(c)] there exists a random variable $Y\in L_P^p(\Omega;V)$
  such that for all (some) orthonormal basis $(e_k)_{k\in\N}$ of $H$:
\begin{align*}
  Y=\sum_{k=1}^\infty Te_k \,\Theta e_k
  \quad\text{in $L^2_P(\Omega;V)$}.
\end{align*}
\end{enumerate}
Often this property is taken as a definition of $\gamma$-radonifying operators in the literature.
\end{remark}

\begin{theorem}\label{th.mradHilbert}
If $V$ is a separable Hilbert space and the cylindrical measure
$\theta$ has weak $p$-th moments for some $p\ge 1$  and is of finite cotype then it follows:
\begin{align*}
  \Rad^p(\theta)
  =\{T\in \L(H,V): \,T\text{ is Hilbert-Schmidt}\}
\end{align*}
\end{theorem}
\begin{proof}
It is well known that in Hilbert spaces the class of $p$-radonifying operators (see proof of Proposition \ref{pro.abs-rad}) coincides with the space of Hilbert-Schmidt operators. Thus, the class of Hilbert-Schmidt operators is a subset of $\Rad^p(\theta)$.

Let $T$ be in $\Rad^p(\theta)$. Then the cylindrical
measure $\theta\circ T^{-1}$ extends to a Radon measure of order $p$ and $\theta\circ
(T^{\ast\ast})^{-1}$ is a $\sigma(V^{\ast\ast},V^\ast)$-Radon measure. Denote the cotype of $\theta$ by $q\in [0,\infty)$. If $q\le p$ then $\theta$ is also of cotype $p$, and thus, we can assume that $q\ge 1$.
Theorem VI.5.9 in \cite{Vaketal} implies that $T^{\ast}$ is $q$-absolutely summing, which is equivalent to the fact that $T^{\ast}$ is Hilbert-Schmidt, since $H$ and $V$ are Hilbert spaces.
\end{proof}

\begin{remark}
  If $\theta$ is a genuine probability measure with $p$-th moments then each operator in $\L(H,V)$ is in $\Rad^p(\theta)$ and Theorem \ref{th.mradHilbert} cannot be true. This case is excluded since in this case $\theta$ cannot be of finite cotype: if $(h_n)_{n\in\N}$ is a sequence in $H$ which converges sequentially weakly to $0$ then Lebesgue's theorem implies
\begin{align*}
\lim_{n\to\infty}  \int_{H}\scapro{h_n}{h}^p\, \theta(dh)=0.
\end{align*}
But if $H$ is infinite dimensional then we can choose $\norm{h_n}=1$ for all $n\in\N$.
\end{remark}

{\bf Stable cylindrical measures:}
A cylindrical measure $\theta$ on $\Z(H)$ is called {\em stable of order $\alpha\in (0,2]$} if there exists a measure space $(S,\mathcal{S},m)$ and a linear bounded operator $F\colon H\to L_m^\alpha(S;\R)$ such that the characteristic function $\phi_\theta$ of
$\theta$ obeys
\begin{align}\label{eq.char-stable}
  \phi_\theta\colon H\to \C,\qquad \phi_\theta(h)=\exp\left(-\norm{Fh}_{L_m^\alpha}^\alpha\right).
\end{align}
The characteristic function of $\theta\circ T^{-1}$ for an arbitrary operator $T\in \L(H,V)$ is given by
\begin{align*}
    \phi_{\theta\circ T^{-1}}\colon V^\ast \to \C,\qquad \phi_{\theta\circ T^{-1}}(v^\ast)=\exp\left(-\norm{(FT^\ast)v^\ast}_{L_m^\alpha}^\alpha\right).
\end{align*}
It follows that $T$ is $\theta$-radonifying if and only if $FT^\ast \in \Lambda_\alpha(V^\ast, L^\alpha_m)$, where
\begin{align*}
&\Lambda_\alpha(V^\ast, L^\alpha_m):=
\Big\{R\in \L(V^\ast,L_m^\alpha):\, v^\ast\mapsto \exp(-\norm{R v^\ast}^\alpha_{L_m^\alpha}) \\
&\hspace*{5cm} \text{ is the characteristic function of a Radon measure on $V$} \Big\}.
\end{align*}
By taking into account that an $\alpha$-stable measure has finite $r$-th moments for all
$r<\alpha$ according to Theorem 3.2 in \cite{Acosta75}, we obtain for each $p<\alpha$ that $T$ is in $\Rad^p(\theta)$ if only if $FT^\ast \in \Lambda_\alpha(V^\ast, L^\alpha_m)$. The spaces $\Lambda_\alpha(V^\ast, L^\alpha_m)$ are surveyed  in \cite[Se.7.8]{Linde}, however an explicit description is only known in a few case.
A case which can be easily described is the following:
\begin{example}
  Assume that $V=\ell^q$ for some $q\in [2,\infty)$ and  $L_m^\alpha(S;\R)=\ell^\alpha$
  for some $\alpha<q^\prime$ where $q^\prime:=q/(q-1)$. Let $\theta$ be an $\alpha$-stable measure on the Hilbert space $H$ with characteristic function of the form \eqref{eq.char-stable}.
    Then an operator $T\in \L(H,V)$ is in $\Rad^p(\theta)$ for $p<\alpha$ if and only if
\begin{align*}
  \sum_{k=1}^\infty  \left( \sum_{j=1}^\infty
   \abs{\scapro{FT^\ast e_j}{e_k}}^{q^\prime} \right)^{\tfrac{\alpha}{q^{\prime}}}<\infty,
\end{align*}
where $(e_k)_{k\in\N}$ denotes the canonical Schauder basis for the spaces of
sequences.
\end{example}

\begin{example}
  Assume that $V$ is given by some $L^q$ space for $q\in [2,\infty)$ and
  $\theta$ is an $\alpha$-stable cylindrical measure on the Hilbert space $H$ with characteristic function of the form \eqref{eq.char-stable}. Then an operator
$T\in \L(H,V)$ is in $\Rad^p(\theta)$ for $p<\alpha$ if and only if
$FT^\ast $ is $r$-absolutely summing for any $r\in (0,q^\prime)$, i.e.\ $FT^\ast\in \Pi^r(L^{q^\prime},L_m^\alpha)$; see Proposition 7.8.7 in \cite{Linde}.
\end{example}
Stable cylindrical probability measures might be considered as a subset of infinitely divisible cylindrical measures with elements which are the most similar ones to the canonical Gaussian cylindrical measure $\gamma$. Nevertheless, many properties known for $\gamma$-radonifying operators do not hold for radonifying operators of stable cylindrical measures. One of these is the ideal property which is true for $\gamma$-radonifying operators:
let $H$ and $H^\prime$ be Hilbert spaces and $V$ and $V^\prime$ be Banach spaces. Then if $T\in \Rad_{H,V}^2(\gamma)$, $S_1\in \L(H^\prime, H)$ and $S_2\in \L(V,V^\prime)$
then $S_2TS_1\in \Rad_{H^\prime, V^\prime}^2(\gamma)$. This result can be found in \cite{vanNeerven10}. However, already for stable cylindrical measures it is known that
the ideal property is not satisfied any more: for $q>2$ there exists
an operator $T\in \Rad^p_{L^p,L^q}(\theta)$ and $S\in \L(L^p,L^p)$ such that $TS$ is not in $\Rad^p_{L^p,L^q}(\theta)$; see \cite{Linde} for this result. \\

{\bf Compound Poisson cylindrical measures:}
 a {\em compound Poisson cylindrical distribution} (see Example 3.5 in \cite{ApplebaumRiedle}) is an infinitely divisible cylindrical measure  $\theta$ on $\Z(H)$ with characteristic function
\begin{align}\label{eq.charcomppoisson}
\phi_{\theta}:H\to\C,\qquad
  \phi_{\theta}(h)=\exp\left(c\int_H \left(e^{i\scapro{h}{g}}-1\right)\, \nu(dg)\right),
\end{align}
where $\nu$ is a cylindrical probability measure on $\Z(H)$ and $c>0$ is a constant.
It follows from the finite dimensional theory of infinite divisible distributions that
$\theta$  has weak second moments if and only if $\nu$ has weak second moments.

Equivalently, one can introduce a compound Poisson cylindrical distribution by cylindrical random variables.
Let $X_1, X_2,\dots $ be independent, cylindrical random variables in $H$ with identical cylindrical distribution $\nu$ and let $N$ be an independent, integer-valued Poisson distributed random variable with intensity $c>0$, all defined on the probability space $(\Omega,\A,P)$. Then
\begin{align*}
 Y\colon H\to L^0_P(\Omega;\R),\qquad
  Yh:=\begin{cases}
    0, & \text{if }N=0,\\
   X_1h+\dots +X_{N}h, &\text{else,}
  \end{cases}
\end{align*}
defines a cylindrical random variable $Y$ with a characteristic function which is of the form \eqref{eq.charcomppoisson}.

\begin{theorem}\label{th.comppoisson}
For a compound Poisson cylindrical distribution $\theta$ with characteristic function \eqref{eq.charcomppoisson} it follows for each $p\ge 1$ that
   \begin{align*}
     \Rad^p (\theta)= \Rad^p(\nu).
   \end{align*}
\end{theorem}
\begin{proof}
  Let $T\in \Rad^p (\theta)$. Then $\mu:=\theta\circ T^{-1}$ is an infinitely divisible  measure on $\Borel(V)$ with $p$-th moment.
  If $\xi$ denotes the L\'evy measure of $\mu$ then it follows that $\xi=c(\nu \circ T^{-1})$ on $\Z(V)$ due to the uniqueness of cylindrical L\'evy measures. Thus, the image cylindrical measure $\nu\circ T^{-1}$ extends to the  probability measure $c^{-1}\xi$.  Let $Y$ be a $V$-valued random variable with distribution $\mu$ and $(X_k)_{k\in\N}$ a family of
  independent, $V$-valued random variables with distribution $c^{-1}\xi$. It follows that
  \begin{align*}
    E\left[\norm{X_1}^p\right]
    \le \frac{e^c}{c}\sum_{k=1}^\infty E\left[\norm{X_1+\dots +X_k}^p\right] \frac{c^k}{k!}e^{-c}
    = \frac{e^c}{c} E\left[\norm{Y}^p\right]<\infty,
  \end{align*}
i.e. the Radon measure $c^{-1}\xi$ has moments of order $p$ which shows $T\in \Rad^p(\nu)$.

If we assume $T\in \Rad^p(\nu)$ then $\nu\circ T^{-1}$ is a probability measure on $\Borel(V)$ and
\begin{align*}
 \mu:\Borel(V)\to [0,1], \qquad
  \mu(C):=e^{-c} \sum_{k=0}^\infty \frac{c^k(\nu\circ T^{-1})^{\ast k}(C)}{k!}
\end{align*}
defines a  probability measure on $\Borel(V)$ with characteristic function
\begin{align*}
\phi_\mu\colon V^\ast\to\C,\qquad  \phi_\mu(v^\ast)= \exp\left(c \int_{V} \left(e^{i\scapro{v}{v^\ast}}-1\right) \, (\nu\circ T^{-1})(dv)\right),
\end{align*}
see \cite[Pro.5.3.1]{Linde}. Since $\phi_{\mu}=\phi_{\theta\circ T^{-1}}$ it follows that $\theta\circ T^{-1}$ extends to
the Radon measure $\mu$ on $\Borel(V)$. As before, let $Y$ and $X_1,X_2,\dots $ denote independent random variables with distributions $\mu$ and $\nu\circ T^{-1}$. The  measure $\mu$ has $p$-th moments since Minkowski's inequality implies
\begin{align*}
  E\left[\norm{Y}^{p}\right]
= \sum_{k=1}^\infty  E\left[\norm{X_1+\dots + X_k}^p\right]
 \frac{c^k}{ k!} e^{-c}
\le \sum_{k=1}^\infty k^p E\left[\norm{X_1}^p\right]
 \frac{c^k}{ k!} e^{-c}
<\infty,
\end{align*}
which shows that  $T\in \Rad^p(\theta)$.
\end{proof}

\begin{example} (Cylindrical normally distributed jumps)\label{ex.cylGauss}\\
Models of share prices perturbed by a discontinuous noise with normally distributed jumps are considered in Financial Mathematics from its very early times; see for example the work \cite{Merton} by Merton. Accordingly, let $\nu$ be the canonical Gaussian cylindrical measure $\gamma$ and let $c>0$ be a constant.
Then the  compound Poisson cylindrical distribution $\theta$   with characteristic function  \eqref{eq.charcomppoisson}
obeys
\begin{align*}
    \Rad^p(\theta)=\Rad^p(\gamma).
\end{align*}
\end{example}

\section{Application: Wiener integrals}\label{se.Application}

In this section we apply the theory of radonifying operators developed above in order to introduce Wiener integrals with respect to a L\'evy process $L$ with weak second moments on a separable Banach space $U$. In fact, the same approach can be applied if $L$ is only a cylindrical L\'evy process, but we want to avoid any more technical complications here; see
\cite{Riedle14b} for details.

Recall that the L\'evy process $L$ can be decomposed into $L(t)=b+ W(t)+ M(t)$ for all $t\ge 0$, where $b\in U$ and $W$ is a Wiener process with a covariance operator $C\in \L(U^\ast,U)$  and $M$ is a L\'evy process with weak second moments and a L\'evy measure $\mu$. The Wiener integrals with respect to the Wiener process $W$ are developed in the publications \cite{BrzezniakvanNeerven00} and \cite{vanNeervenWeis} with many sophisticated refinements and applied to the stochastic Cauchy problem. Our rather simplified presentation  below for integration with respect to $W$ illustrates the core idea in the approach developed in \cite{vanNeervenWeis}. In our work \cite{OnnoMarkus}, we extend the approach in \cite{vanNeervenWeis} to develop a Wiener integral with respect to a martingale-valued measure. By the theory developed here,
we are able to relate  this integral to $\theta$-radonifying operators.

We begin with defining a Wiener integral with respect to the discontinuous martingale $M$ with L\'evy measure $\mu$. For $\rho:=\lambda\otimes\mu$, where $\lambda$ denotes the Lebesgue  measure on $[0,T]$, define $H_M:=L^2_\rho([0,T]\times U;\R)$. Let $V$ denote another separable Banach space and let $F\colon [0,T]\to \L(U,V)$ be a function satisfying
\begin{align}\label{eq.F-weakM2}
  \scapro{F(\cdot)\cdot}{v^\ast}\in H_M
  \qquad\text{for all }v^\ast\in V^\ast.
\end{align}
Then one can define a cylindrical random variable by
\begin{align*}
  I_M\colon V^\ast\to L^2_P(\Omega;\R),\qquad
   I_M v^\ast=\int_0^T F^\ast(s)v^\ast\, dM(s).
\end{align*}
Since the integrand is $U^\ast$-valued and $M$ has weak second moments,  the integral can be easily defined by following an It{\^o} approach, see e.g.\ \cite{Riedle14}, or by the approach of M\'etivier and Pellaumail in \cite{Metivier}, or as introduced by Rosi\'nski  in \cite{Rosinski87}. In all cases, it follows that $I_M $ is an infinitely divisible cylindrical random variable with weak second moments. The covariance operator of $I_M$ is given by
\begin{align*}
  Q_M\colon V^\ast\to V, \qquad
   \scapro{Q_Mv^\ast}{w^\ast}=\int_{[0,T]\times U} \scapro{F(s)u}{v^\ast}\scapro{F(s)u}{w^\ast}\,\rho(ds,du).
\end{align*}
The mapping $Q_M$ is $V$-valued and not only $V^{\ast\ast}$-valued, since Pettis' measurability theorem guarantees due to \eqref{eq.F-weakM2} that $(t,u)\mapsto F(t)u$ is strongly measurable. The covariance operator
$Q_M$ can be factorised by
\begin{align*}
 j_M\colon H_M\to V,\qquad \scapro{j_M f}{v^\ast}
  :=\int_{[0,T]\times U}\scapro{F(s)u}{v^\ast} f(s,u)\,\rho(ds,du).
\end{align*}
Since the adjoint operator is given by $j_M^\ast w^\ast = \scapro{F(\cdot)\cdot}{w^\ast}
=F^\ast(\cdot)w^\ast$ for all $w^\ast\in V^\ast$ it follows that
$Q_M=j_Mj_M^\ast$, and thus $H_M=L^2_\rho([0,T]\times U;\R)$ is established as
the reproducing kernel Hilbert space of $Q_M$. Define the cylindrical random variable
\begin{align*}
  \Theta_M\colon H_M\to L^2_P(\Omega;\R),
  \qquad \Theta_M f=\int_{[0,T]\times U} f(s,u)\, M(ds,du),
\end{align*}
and let $\theta_M$ denote the cylindrical distribution of $\Theta_M$.
Let $(f_k)_{k\in\N}$ be an orthonormal basis of $H_M$ and choose
$v_k^\ast\in V^\ast$ such that $f_k=j_M^\ast v_k^\ast$.
By continuity of $\Theta_M$ it follows for all $v^\ast\in V^\ast$ that
we have in $L^2_P(\Omega;\R)$:
\begin{align*}
I_M v^\ast = \Theta_M (j_M^\ast v^\ast) =
\Theta_M \left(\sum_{k=1}^\infty \scapro{f_k}{j_M^\ast v^\ast}f_k\right)
=\sum_{k=1}^\infty \scapro{j_M f_k}{v^\ast} \Theta_M f_k
=\sum_{k=1}^\infty \scapro{j_M f_k}{v^\ast} I_M v_k^\ast.
\end{align*}
In summary, we have explicitly derived the setting of Section \ref{se.infdiv}: for the cylindrical random variable $I_M$ we derived the reproducing kernel Hilbert space $H_M$
of its covariance operator $Q_M$ with embedding $j_M\colon H_M\to V$. In addition, we constructed the cylindrical random variable $\Theta_M$ in $H_M$, which is based on the Karhunen-Lo{\`e}ve representation of $I_M$ according to \eqref{eq.Loeve} and which satisfies $j_M(\Theta_M)=I_M$.

It follows from Lemma \ref{le.properties} that $\Theta_M$ is an infinitely divisible cylindrical random variable, and by approximating $f\in H_M$ by step functions, the characteristic function of $\Theta_M$ is given by
\begin{align*}
 \phi_{\Theta_M}\colon H_M\to \C, \qquad \phi_{\Theta_M}(f)= \exp\left(\int_{H_M} \left(e^{i\scapro{f}{g}}-1-i\scapro{f}{g}\right)
  \, \nu(dg)\right),
\end{align*}
where $\nu$ is a cylindrical measure on $\Z(H_M)$ satisfying
$\nu\circ (\scapro{\cdot}{f})^{-1}=\rho\circ f^{-1}$ for all $f\in H_M$.

Since  the cylindrical distribution of $I_M$ equals $\theta_M\circ j_M^{-1}$ according to Lemma \ref{le.properties}, there exists a random variable $Y_M\in L^0_P(\Omega;V)$ obeying
\begin{align}\label{eq.int-M}
    \scapro{Y_M}{v^\ast}= \int_0^T F^\ast(s)v^\ast\,M(ds)
    \qquad\text{for all }v^\ast\in V^\ast,
\end{align}
if and only if $j_M\colon H_M\to V$ is $\theta_M$-radonifying.

The same approach can be applied to introduce the stochastic integral with respect to the Wiener process $W$
with covariance operator $C$. Let $K$ denote the reproducing kernel Hilbert space of $C$ with embedding $i_C\colon K\to U$, i.e.\ $C=i_Ci_C^\ast$. For functions $F\colon [0,T]\to \L(U,V)$ satisfying
\begin{align}\label{eq.F-weakW2}
   i_C^\ast F^\ast(\cdot)v^\ast\in L^2([0,T]; K)
  \qquad\text{for all }v^\ast\in V^\ast,
\end{align}
one can define a cylindrical random variable by
\begin{align*}
  I_W\colon V^\ast\to L^2_P(\Omega;\R),\qquad
   I_W v^\ast=\int_0^T F^\ast(s)v^\ast\, dW(s).
\end{align*}
The covariance operator of $I_W$ is given by
\begin{align*}
  Q_W\colon V^\ast\to V,\qquad
  \scapro{Q_Wv^\ast}{w^\ast}=\int_0^T \scapro{i_C^\ast F^\ast(s)v^\ast}{i_C^\ast F^\ast(s)w^\ast}\, ds.
\end{align*}
 The covariance operator $Q_W$ can be factorised through the Hilbert space
$H_W:=L^2([0,T];K)$, and the embedding is given by
\begin{align*}
  j_W\colon H_W\to V, \qquad
  \scapro{j_Wf}{v^\ast}=\int_0^T \scapro{i_C^\ast F^\ast(s)v^\ast}{f(s)}\, ds,
\end{align*}
with adjoint operator $j_W^\ast v^\ast= i_C^\ast F^\ast(\cdot)v^\ast$. If $\gamma$ denotes the canonical Gaussian cylindrical measure on $H_W$ it follows that $\gamma\circ j_W^{-1}$ is a cylindrical Gaussian distribution with covariance operator $j_Wj_W^\ast$, that is
$\gamma\circ j_W^{-1}$ coincides with the cylindrical distribution of $I_W$. Consequently, we obtain that there exists a random variable $Y_W\in L^0_P(\Omega;V)$ obeying
\begin{align*}
  \scapro{Y_W}{v^\ast}=\int_0^T F^\ast(s)v^\ast\, dW(s) \qquad \text{for all }v^\ast\in V^\ast,
\end{align*}
if and only if $j_W\colon H_W\to V$ is $\gamma$-radonifying.

Finally, let $F\colon [0,T]\to \L(U,V)$ be a function obeying \eqref{eq.F-weakM2},
\eqref{eq.F-weakW2} and
\begin{align}\label{eq.F-weakPettis}
  \scapro{F^\ast(\cdot)v^\ast}{b}\in L^1([0,T];\R)\qquad\text{for all $v^\ast \in V^\ast$.}
\end{align}
Then one can define for each $A\in\Borel([0,T])$ a cylindrical random variable
$I_A\colon V^\ast \to L^2_P(\Omega;\R)$   by
\begin{align}
I_A v^\ast
  = \int_0^T \1_A(s) \scapro{F^\ast(s)v^\ast}{b} \,ds +   \int_0^T \1_A(s) F^\ast(s)v^\ast\,dW(s) + \int_0^T \1_A(s) F^\ast(s)v^\ast \,dM(s).
\end{align}
The function $F$ is called {\em stochastically integrable  with respect to $L$} if and only if for each $A\in\Borel([0,T])$ there exists a random variable $Y_A\in L^0_P(\Omega;V)$ such that
\begin{align}\label{eq.def-int}
  \scapro{Y_A}{v^\ast}=I_Av^\ast \qquad \text{for all }v^\ast\in V^\ast.
\end{align}
By the derivation above one obtains the following result:
\begin{theorem}
A function $F\colon [0,T]\to L(U,V)$ satisfying \eqref{eq.F-weakM2},
\eqref{eq.F-weakW2} and \eqref{eq.F-weakPettis}
is stochastically integrable with respect to $L(\cdot)=b+W(\cdot)+M(\cdot)$  if and only if the following are satisfied:
  \begin{enumerate}
    \item[\rm (i)] $F(\cdot)b\colon [0,T]\to V$ is Pettis integrable;
    \item[\rm (ii)] $j_W\colon H_W\to V$ is $\gamma$-radonifying;
    \item[\rm  (iii)] $j_M\colon H_M\to V$ is $\theta_M$-radonifying.
  \end{enumerate}
\end{theorem}
\begin{proof}
{\em If part:} for $A\in \Borel([0,T])$ define the cylindrical random variables
\begin{align*}
&I_b^A\colon V^\ast \to \R, \qquad I_b^Av^\ast=\int_A \scapro{v^\ast}{F(s)b}\, ds, \\
&I_W^A\colon V^\ast\to L^2_P(\Omega;\R),\qquad  I_W^A v^\ast=\int_0^T \1_{A}(s) F^\ast(s)v^\ast\, dW(s).\\
&  I_M^A\colon V^\ast\to L^2_P(\Omega;\R),\qquad I_M^A v^\ast=\int_0^T \1_{A}(s) F^\ast(s)v^\ast\, dM(s).
\end{align*}
We have to show that these cylindrical random variables are induced by genuine random variables
in $L_P^0(\Omega;V)$, respectively.
Condition (i) implies that there exists $v^A\in V$ such that
\begin{align}\label{eq.proof-int1}
  \scapro{v^A}{v^\ast}=I_b^A v^\ast\qquad\text{for all }v^\ast\in V^\ast.
\end{align}
The covariance operator of the cylindrical random variable $I_W^A$ is given by
\begin{align*}
  Q_W^A\colon V^\ast \to V, \qquad
    \scapro{Q_Wv^\ast}{w^\ast}=\int_A \scapro{i_C^\ast F^\ast(s)v^\ast}{i_C^\ast F^\ast (s)w^\ast}\, ds.
\end{align*}
It follows that
\begin{align*}
  \scapro{Q_W^Av^\ast}{v^\ast}\le \scapro{Q_W^{[0,T]}v^\ast}{v^\ast}
\qquad\text{for all }v^\ast\in V^\ast.
\end{align*}
Since Condition (ii) guarantees that $Q_W^{[0,T]}$ is the covariance operator of a Gaussian measure on $\Borel(V)$, Theorem 3.3.1 in \cite{Bogachev} implies that there exists a Gaussian measure on $\Borel(V)$ with
covariance operator $Q_W^A$. Thus, Theorem IV.2.5 in \cite{Vaketal} guarantees that there exists a random variable $Y_W^A\in L^0_P(\Omega;V)$ such that
\begin{align}\label{eq.proof-int2}
  \scapro{Y_W^A}{v^\ast}=I_W^Av^\ast \qquad\text{for all }v^\ast\in V^\ast.
\end{align}
From the independent increments of $M$ it follows that the cylindrical distribution of $I_M^A$ is infinitely divisible and that its cylindrical L\'evy measure $\nu_A$ is given by
\begin{align*}
  \nu_A\colon \Z(V)\to [0,\infty], \qquad
  \nu_A(B)=\int_{A\times U} \1_{B}(F(s)u)\,\rho(ds,du).
\end{align*}
Condition (iii) implies that $\nu_{[0,T]}$ is the genuine L\'evy measure of an infinitely divisible probability measure on $\Borel(V)$. Since
\begin{align*}
  \nu_A(B)\le \nu_{[0,T]}(B)\qquad\text{for all }B\in \Z(V),
\end{align*}
Theorem 3.4 in \cite{Riedle14b} implies that $\nu_A$ extends to a genuine L\'evy measure on $\Borel(V)$.
Theorem IV.2.5 in \cite{Vaketal} guarantees that there exists a random variable $Y_M^A\in L^0_P(\Omega;V)$ such that
\begin{align}\label{eq.proof-int3}
  \scapro{Y_M^A}{v^\ast}=I_M^Av^\ast \qquad\text{for all }v^\ast\in V^\ast.
\end{align}
It follows from \eqref{eq.proof-int1}, \eqref{eq.proof-int2} and \eqref{eq.proof-int3} that the
random variable $Y_A:=v^A + Y_W^A + Y_M^A$ satisfies \eqref{eq.def-int}.

{\em Only if part:} Let $I_b$, $I_W$ and $I_M$ denote the cylindrical random variables defined in the beginning of the proof for $A=[0,T]$. Stochastic integrability of $F$ implies that there exists a random variable $Y\in L^0_P(\Omega;V)$ such that
\begin{align*}
  \scapro{Y}{v^\ast}=I_b v^\ast + I_W v^\ast + I_M v^\ast \qquad\text{for all }v^\ast\in V^\ast.
\end{align*}
Since $I_W v^\ast$ is symmetric and independent of $(I_b + I_M)v^\ast$ for all $v^\ast\in V^\ast$,
it follows from Proposition 7.14.51 in \cite{Bogachev-Measure} that the cylindrical distributions of
$I_W$ and $I_b+I_M$ extend to probability measures on $\Borel(V)$. Consequently, $j_W$ is $\gamma$-radonifying and $I_b+I_M$ is induced by a random variable in $L^0_P(\Omega;V)$.
Let $M^\prime$ denote an independent copy of $M$. Then $I_b+I_M -I_b-I_{M^\prime}$
is induced by a  random variable $X\in L^0_P(\Omega;V)$ and $X$ is infinitely divisible.
Denoting the L\'evy measure of $X$ by $\xi$ it follows that
\begin{align*}
  \nu_{[0,T]}(B)\le \nu_{[0,T]}(B)+\nu_{[0,T]}(-B)=\xi(B) \qquad\text{for all }B\in \Z(V).
\end{align*}
Theorem 3.4 in \cite{Riedle14b} implies that $\nu_{[0,T]}$ extends to a genuine L\'evy measure on $\Borel(V)$, which yields that $j_M$ is $\theta_M$-radonifying.
\end{proof}

\vspace{10pt}
\noindent\textbf{Acknowledgments:} Theorem \ref{th.Banach} is based on unpublished discussions with O. van Gaans (Leiden). The author would like to thank Kai K{\"u}mmel (Jena) for  careful reading.


\end{document}